\documentclass[ejs]{imsart}

\arxiv{math.PR/0000000}

\RequirePackage{amsthm,amsmath}
\usepackage{amsfonts}   
\usepackage{amssymb}    

\startlocaldefs
\numberwithin{equation}{section}
\input aad.tex

\def\ep{\mathbb{G}}
\endlocaldefs

\begin{document}

\begin{frontmatter}
\title{A Local Maximal Inequality under Uniform Entropy}
\runtitle{Maximal Inequality}

\begin{aug}
\author{Aad van der Vaart\ead[label=e1]{aad@cs.vu.nl}}
\address{Department of Mathematics, Faculty of Sciences, Vrije Universiteit
                   De Boelelaan 1081a, 1081 HV Amsterdam,\\ \printead{e1}}

\author{Jon A. Wellner\thanksref{t2}\ead[label=e2]{jaw@stat.washington.edu}}
\thankstext{t2}{Supported in part by NSF Grant DMS-0804587, and by  NI-AID grant 2R01 AI291968-04}
\address{Department of Statistics, University of Washington, Seattle, WA  98195-4322,\\ \printead{e2}}

\runauthor{van der Vaart and Wellner}
\end{aug}

\begin{abstract}
We derive an upper bound for the mean of the supremum
of the empirical process indexed by a class of functions
that are known to have variance bounded by a small constant $\d$.
The bound is expressed in the uniform entropy integral of
the class at $\d$. The bound yields a
rate of convergence of minimum contrast estimators when
applied to the modulus of continuity of the contrast functions.
\end{abstract}



\end{frontmatter}

\section{Introduction}
The \emph{empirical measure} $\PP_n$ and \emph{empirical process} $\ep_n$ 
of a sample of observations $X_1,\ldots, X_n$ from a probability
measure $P$ on a measurable space $(\X,\A)$ attach to a 
given measurable function $f: \X\to\RR$ the numbers
$$\PP_nf=\frac 1n\sumin f(X_i),
\qquad \ep_nf=\frac1{\sqrt n}\sumin \bigl(f(X_i)-Pf\bigr).$$
It is often useful to study the suprema of these stochastic processes 
over a given class $\F$ of measurable functions. The distribution
of the supremum 
$$\|\ep_n\|_\F:=\sup_{f\in\F}|\ep_nf|$$
is known to concentrate near its mean value, at a rate depending on the
size of the envelope function of the class $\F$, but irrespective of its
complexity. On the other hand, the mean value 
of $\|\ep_n\|_\F$ depends on the size of the class $\F$.
Entropy integrals, of which there are two basic versions,
are useful tools to bound this mean value.

The \emph{uniform entropy integral} was
introduced in \cite{Pollard} and \cite{Kolchinskii}, following \cite{Dudley},
in their study of the abstract version of Donsker's theorem. 
We  define an $L_r$-version of it as 
$$J(\d,\F,L_r)
=\sup_Q \int_0^\d\sqrt{1+\log N\bigl(\e\|F\|_{Q,r},\F,L_r(Q)\bigr)}\,d\e.$$
Here  the supremum is taken over all finitely discrete probability
distributions $Q$ on $(\X,\A)$, the \emph{covering number}
$N\bigl(\e,\F,L_r(Q)\bigr)$ is the minimal
number of balls of radius $\e$ in $L_r(Q)$ needed to cover $\F$,
$F$ is an envelope function of $\F$,
and $\|f\|_{Q,r}$ denotes the norm of a function $f$ in $L_r(Q)$.
The integral is defined relative to an \emph{envelope function}, which need
not be the minimal one, but can be any  
measurable function $F: \X\to\RR$ such that $|f|\le F$ for every $f\in\F$.
If multiple envelope functions are under consideration, then we write
$J(\d,\F\given F,L_r)$ to stress this dependence.
An inequality, due to Pollard (also see \cite{vdVW}, 2.14.1), says,
under some measurability assumptions, that
\begin{equation}
\label{vdVW2.14.1}
\E_P^* \|\ep_n\|_\F \lesssim J(1,\F,L_2)\, \|F\|_{P,2}.
\end{equation}
Here $\lesssim$ means smaller than up to a universal constant.
This shows that for a class $\F$ with finite
uniform entropy integral, the supremum $\|\ep_n\|_\F$ is not essentially
bigger than a multiple of 
the empirical process $\ep_nF$ at the envelope function $F$.
The inequality is particularly useful if this envelope function is
small.

The \emph{bracketing entropy integral} has its roots in the Donsker
theorem of \cite{Ossiander}, again following initial work by Dudley.
For a given norm it can be defined as
$$J_{[\,]}\bigl(\d,\F,\|\cdot\|\bigr)
=\int_0^\d\sqrt{1+\log N_{[\,]}\bigl(\e\|F\|,\F,\|\cdot\|\bigr)}\,d\e.$$
Here the \emph{bracketing number} $N_{[\,]}\bigl(\e,\F,\|\cdot\|\bigr)$
is the minimal number of brackets $[l,u]=\{f:\X\to \RR: l\le f\le u\}$
of size $\|u-l\|$ smaller than $\e$ needed to cover $\F$. A
useful inequality, due to Pollard (also see \cite{vdVW}, 2.14.2), is 
\begin{equation}
\label{vdVW2.14.2}
\E_P^*\|\ep_n\|_\F
\lesssim J_{[\,]}\bigl(1,\F,L_2(P)\bigr)\,\|F\|_{P,2}.
\end{equation}
Bracketing numbers are bigger than covering numbers (at twice the size),
and hence the bracketing integral is bigger than a multiple of the corresponding
entropy integral. However, the bracketing integral involves
only the single distribution $P$, whereas the uniform entropy integral
takes a supremum over all (discrete) distributions, making the
two integrals incomparable in general. Apart from
this difference the two maximal inequalities have the same message.

The two inequalities (\ref{vdVW2.14.1}) and (\ref{vdVW2.14.2})
involve the size of the envelope function,
but not the sizes of the individual functions in the class $\F$. 
They also exploit finiteness of the entropy
integrals only, roughly requiring that the entropy grows at smaller
order than $\e^{-2}$ as $\e\da0$, and not the precise size
of the entropy. In the case of the bracketing integral
this is remedied in the equality
(see \cite{vdVW}, 3.4.2), valid for any class
of functions $f: \X\to [-1,1]$ with $Pf^2\le \d^2PF^2$ and any $\d\in(0,1)$,
\begin{equation}
\label{vdVW3.4.2}
\E_P^*\|\ep_n\|_\F
\lesssim J_{[\,]}\bigl(\d,\F,L_2(P)\bigr)\,\|F\|_{P,2}
\ \biggl(1+\frac{J_{[\,]}\bigl(\d,\F,L_2(P)\bigr)}{\d^2\sqrt n\|F\|_{P,2}} 
\biggr).
\end{equation}
Here the assumption that the class of functions
is uniformly bounded is too restrictive for some applications,
but can be removed if the entropy integral is computed relative
to the stronger ``norm'' 
$$\|f\|_{P,B}=\Bigl(2P\bigl(e^{|f|}-1-|f|\bigr)\Bigr)^{1/2}.$$
%
Although it is not a norm, this quantity can be used to
define the size of brackets and hence bracketing numbers.
Inequality (\ref{vdVW3.4.2}) is valid for an arbitrary
class of functions with $\|f\|_{P,B}\le \d\|F\|_{P,B}$ 
if the $L_2(P)$-norm is replaced by $\|\cdot\|_{P,B}$ in its right
side (at four appearances) (see Theorem~3.4.3 of \cite{vdVW}).
The ``norm'' $\|\cdot\|_{P,B}$ derives from the refined version of Bernstein's
inequality, which was first used in the literature on rates of convergence of 
minimum contrast estimators in \cite{BirgeMassart} (also see~\cite{vdGeer}).

Maximal inequalities of type (\ref{vdVW3.4.2}) using \emph{uniform entropy} 
are thus far unavailable.  In
this note we derive an exact parallel of (\ref{vdVW3.4.2}) for
uniformly bounded functions, and investigate similar inequalities for
unbounded functions. The validity of these results seems
unexpected, as the stronger control given by bracketing 
has often been thought necessary for estimates of moduli
of continuity. It was suggested to us by Theorem~3.1 and its proof
in \cite{Gine}.

\subsection{Application to minimum contrast estimators}
Inequalities involving the sizes of the functions $f$ 
are of particular interest in the investigation
of empirical minimum contrast estimators. Suppose that $\hat\q_n$ mimimizes
a criterion of the type
$$\q\mapsto \PP_nm_\q,$$
for given measurable functions $m_\q: \X\to\RR$ indexed
by a parameter $\q$, and that the
population contrast satisfies,
for a ``true'' parameter $\q_0$ and some metric $d$ on the parameter set,
$$Pm_\q-P m_{\q_0}\gtrsim d^2(\q,\q_0).$$
A bound on the rate of convergence of $\hat\q_n$ to $\q_0$ can then be
derived from the modulus of continuity of
the empirical process $\ep_nm_\q$ indexed by the functions $m_\q$.
Specifically (see e.g.\ \cite{vdVW}, 3.2.5) if $\phi_n$ 
is a function such that $\d\mapsto\phi_n(\d)/\d^\a$ is decreasing for
some $\a<2$ and
\begin{equation}
\label{EqModulusIN}
\E \sup_{\q: d(\q,\q_0)<\d} \bigl|\ep_n(m_\q-m_{\q_0})\bigr|
\lesssim \phi_n(\d),
\end{equation}
then $d(\hat\q_n,\q_0)=O_P(\d_n)$, for $\d_n$ any solution to
\begin{equation}
\label{EqROC}
\phi_n(\d_n)\le \sqrt n\d_n^2.
\end{equation}
Inequality (\ref{EqModulusIN}) involves the empirical process
indexed by the class of functions $\M_\d=\{m_\q-m_{\q_0}: d(\q,\q_0)<\d\}$.
If $d$ dominates the $L_2(P)$-norm, or another norm $\|\cdot\|$
that can be used in an equality of the
type (\ref{vdVW3.4.2}),  such as
the Bernstein norm, and the norms of the envelopes of the 
classes $M_\d$ are bounded in $\d$, then we can choose 
$$\phi_n(\d)=J\bigl(\d,\M_\d,\|\cdot\|\bigr)
\ \biggl(1+\frac{J\bigl(\d,\M_\d,\|\cdot\|\bigr)}{\d^2\sqrt n}\biggr),
$$
where $J$ is an appropriate entropy integral.
For this choice the inequality (\ref{EqROC}) is equivalent to 
\begin{equation}
\label{EqRateEI}
J\bigl(\d_n,\M_{\d_n},\|\cdot\|\bigr)\le \sqrt n\d_n^2.
\end{equation}
Thus a rate of convergence can be read off directly from the
entropy integral. 

We note that an inequality of type (\ref{vdVW3.4.2})
is unattractive for very small $\d$, as the bound may even increase
to infinity as $\d\da0$. However, it is
accurate for the range of $\d$ that
are important in the application to moduli of continuity.

Moduli of continuity also play an important role in model
selection theorems. See for instance \cite{Massart}.

Inequalities involving uniform entropy permit for instance the immediate
derivation of rates of convergence for minimum contrast functions
that form VC-classes. Furthermore, uniform entropy is preserved
under various (combinatorial) operations to make new classes of functions.
This makes uniform entropy integrals a useful tool in situations
where  bracketing numbers may be difficult to handle. Equation
(\ref{EqRateEI}) gives an elegant characterization of rates
of convergence in these situations, where thus far ad-hoc arguments 
were necessary.

\section{Uniformly Bounded Classes}
Call the class $\F$ of functions  $P$-\emph{measurable} if the map
$$(X_1,\dots, X_n)\mapsto \sup_{f\in\F}\Bigl|\sumin e_if(X_i)\Bigr|$$
on the completion of the probability space $(\X^n,\A^n,P^n)$ is measurable,
for every sequence $e_1,e_2,\ldots, e_n\in \{-1,1\}$.

\begin{theorem}
\label{TheoremGineKoltchinskiiAbstract}
Let $\F$ be a $P$-measurable class of measurable functions 
with envelope function $F\le1$ and such 
that $\F^2$ is $P$-measurable. 
If $Pf^2< \d^2PF^2$, for every $f$ and some $\d\in(0,1)$, then
$$\E_P^*\|\ep_n\|_\F
\lesssim J\bigl(\d,\F,L_2\bigr)
\Bigl(1+\frac{J(\d,\F,L_2)}{\d^2\sqrt n\|F\|_{P,2}}\Bigr)\|F\|_{P,2}.$$
\end{theorem}

\begin{proof}
We use the following refinement of (\ref{vdVW2.14.1})
(see e.g.\ \cite{vdVW}, 2.14.1): for
any $P$-measurable class $\F$, 
\begin{equation}
\label{EqRef2.14.1}
\E^*_P\|\ep_n\|_\F\lesssim \E_P^* J\Bigl(\frac{\sup_f (\PP_nf^2)^{1/2}}
{(\PP_nF^2)^{1/2}},\F,L_2\Bigr)\,(\PP_nF^2)^{1/2}.
\end{equation}
Because $\d\mapsto J(\d,\F,L_2)$ is the integral of 
a nonincreasing nonnegative function, it is a concave function
such that the map $t\mapsto J(t)/t$, which is the average
of its derivative over $[0,t]$, is nonincreasing. The concavity
shows that its \emph{perspective} $(x,t)\mapsto t J(x/t,\F,L_2)$ is
a concave function of its two arguments (cf. \cite{Boyd}, page 89).
Furthermore, the ``extended-value extension'' of this function 
(which by definition is $-\infty$ if $x\le0$ or $t\le 0$)
is obviously nondecreasing in its
first argument and was noted to be nondecreasing in its second argument.
Therefore, by the vector composition rules for concave functions 
(\cite{Boyd}, pages 83--87, especially lines -2 and -1 of page 86),
the function $(x,y)\mapsto H(x,y):=J\bigl(\sqrt{x/y},\F,L_2\bigr) \sqrt y$
is concave. We have that $\E_P^* \PP_nF^2=\|F\|_{P,2}^2$.
Therefore, by an application of Jensen's inequality
to the right side of the preceding display we obtain,
for $\s_n^2=\sup_f \PP_nf^2$,
\begin{equation}
\label{EqHulp}
\E_P^*\|\ep_n\|_\F
\lesssim J\Bigl(\frac{\sqrt{\E_P^* \s_n^2}}{\|F\|_{P,2}},\F,L_2\Bigr)\|F\|_{P,2}.
\end{equation}
The application of Jensen's inequality with outer expectations
can be justified here by the monotonicity of 
the function $H$, which shows that the measurable majorant of 
a variable $H(U,V)$ is bounded above by $H(U^*,V^*)$,  for
$U*$ and $V^*$ measurable majorants of $U$ and $V$.
Thus $\E^*H(U,V)\le \E H(U^*,V^*)$, after which Jensen's
inequality can be applied in its usual (measurable) form.

The second step of the proof is to bound $\E_P^*\s_n^2$.
Because $\PP_nf^2=Pf^2+n^{-1/2}\ep_nf^2$ and $Pf^2\le \d^2PF^2$
for every $f$, we have 
\begin{equation}
\label{EqForTemporaryUse}
\E_P^*\s_n^2\le \d^2\|F\|_{P,2}^2+\frac1{\sqrt n}\E_P^*\|\ep_n\|_{\F^2}.
\end{equation}
Here the empirical process in the second term can be replaced
by the symmetrized empirical process $\ep_n^o$ (defined
as $\ep_n^of=n^{-1/2}\sumin \e_if(X_i)$ for independent Rademacher
variables $\e_1,\e_2,\ldots,\e_n$) at the cost
of adding a multiplicative factor 2 (e.g.\ \cite{vdVW},
 2.3.1). The expectation can be factorized
as the expectation on the Rademacher variables $\e$ followed
by the expectation on $X_1,\ldots,X_n$, and $\E_\e\|\ep_n^o\|_{\F^2}
\le 2 \E_\e\|\ep_n^o\|_{\F}$ by the contraction principle for Rademacher
variables (\cite{Ledoux}, Theorem~4.12), 
and the fact that $F\le 1$ by assumption. Taking the expectation
on $X_1,\ldots,X_n$, we obtain that 
$\E_P^*\|\ep_n\|_{\F^2}\le 4\E_P^*\|\ep_n^o\|_\F$, which in turn is bounded
above by $8\E_P^*\|\ep_n\|_\F$ by the
 desymmetrization inequality (e.g.\ 2.36 in \cite{vdVW}).

Thus $\F^2$ in the last term of (\ref{EqForTemporaryUse}) can be 
replaced by $\F$, at the cost of inserting a constant.
Next we apply (\ref{EqHulp}) to this term,
and conclude that $z^2:=\E_P^* \s_n^2/\|F\|_{P,2}^2$ satisfies the inequality
\begin{equation}
\label{EqBoundz}
z^2\lesssim \d^2+\frac {J(z,\F,L_2)}{\sqrt n\|F\|_{P,2}}.
\end{equation}
We apply Lemma~\ref{LemmaRecursion} with $r=1$, $A=\d$ and
$B^2=1/(\sqrt n\|F\|_{P,2})$ to see that
$$J(z,\F,L_2)
\lesssim J(\d,\F,L_2)+ \frac{J^2(\d,\F,L_2)}{\d^2\sqrt n\|F\|_{P,2}}.$$
We insert this in (\ref{EqHulp}) to complete the proof.
\end{proof}

\begin{lemma}
\label{LemmaRecursion}
Let $J:(0,\infty)\to\RR$ be a concave, nondecreasing function with $J(0)=0$.
If $z^2\le A^2+B^2 J(z^r)$ for some $r\in(0,2)$ and $A,B>0$, then
$$J(z)\lesssim J(A)\Bigl[1+J(A^r)\Bigl(\frac BA\Bigr)^2\Bigr]^{1/(2-r)}.$$
\end{lemma}

\begin{proof}
For $t>s>0$ we can write $s$ as the convex combination $s=(s/t)t+(1-s/t)0$ of
$t$ and $0$. Since $J(0)=0$, the concavity of $J$ gives
that $J(s)\ge (s/t) J(t)$. Thus the function $t\mapsto J(t)/t$ is 
decreasing, which implies that $J(Ct)\le C J(t)$ for $C\ge1$ and
any $t>0$.

By the monotonicity of $J$ and the assumption on $z$ it follows that
$$J(z^r)\le J\Bigl( \bigl(A^2+B^2J(z^r)\bigr)^{r/2}\Bigr)
\le J(A^r)\Bigl( 1+\Bigl(\frac BA\Bigr)^2J(z^r)\Bigr)^{r/2}.$$
This implies that $J(z^r)$ is bounded by a multiple of
the maximum of $J(A^r)$ and $J(A^r)(B/A)^rJ(z^r)^{r/2}$.
If it is bounded by the second one, then $J(z^r)^{1-r/2}\lesssim
J(A^r)(B/A)^r$. We conclude that
$$J(z^r)\lesssim J(A^r)+J(A^r)^{2/(2-r)}\Bigl(\frac BA\Bigr)^{2r/(2-r)}.$$
Next again by the monotonicity of $J$,
\begin{align*}
J(z)&\le J\Bigl(\sqrt{A^2+B^2J(z^r)}\Bigr)
\le J(A)\sqrt{1+\Bigl(\frac BA\Bigr)^2J(z^r)}\\
&\lesssim J(A)\Bigl[1+\Bigl(\frac BA\Bigr)^2
\Bigl(J(A^r)+J(A^r)^{2/(2-r)}\Bigl(\frac BA\Bigr)^{2r/(2-r)}\Bigr)\Bigr]^{1/2}\\
&\lesssim J(A)\Bigl[1+\sqrt{J(A^r)}\Bigl(\frac BA\Bigr)
+\Bigl(\frac BA\Bigr)^{2/(2-r)}J(A^r)^{1/(2-r)}\Bigr].
\end{align*}
The middle term on the right side is bounded by a multiple of
the sum of the first and third terms, 
since $x\lesssim 1^p+x^q$ for any conjugate pair $(p,q)$ and any $x>0$,
in particular $x=\sqrt{J(A^r)}B/A$.
\end{proof}

For values of $\d$ such that $\d\|F\|_{P,2}\ll 1/\sqrt n$ 
Theorem~\ref{TheoremGineKoltchinskiiAbstract} can be improved. 
(This seems not to be of prime interest for statistical applications.) Its
bound can be written in the form 
$J(\d,\F,L_2)\|F\|_{P,2}+J^2(\d,\F,L_2)/(\d^2\sqrt n)$. In the second
term $\d$ can be replaced by $1/(\|F\|_{P,2}\sqrt n)$, which is
better if $\d$ is smaller than the latter number,
as the function $\d\mapsto J(\d,\F,L_2)/\d$ is decreasing.

\begin{lemma}
Under the conditions of 
Theorem~\ref{TheoremGineKoltchinskiiAbstract},
$$\E_P^*\|\ep_n\|_\F
\lesssim J\bigl(\d,\F,L_2\Bigr)\|F\|_{P,2}+
J^2\Bigl(\frac1{\sqrt n\|F\|_{P,2}},\F,L_2\Bigr)\sqrt n\|F\|_{P,2}^2.$$
\end{lemma}

\begin{proof}
We follow the proof of Theorem~\ref{TheoremGineKoltchinskiiAbstract}
up to (\ref{EqBoundz}), but next use the alternative bounds
\begin{align*}
J(z,\F,L_2)
&\lesssim 
J\Bigl(\sqrt{\d^2+\frac {J(z,\F,L_2)}{\sqrt n\|F\|_{P,2}}},\F,L_2\Bigr)\cr
&\le J(\d,\F,L_2)
+J\Bigl(\sqrt{\frac {J(z,\F,L_2)}{\sqrt n\|F\|_{P,2}}},\F,L_2\Bigr)\cr
&\le J(\d,\F,L_2)+J(\d_n,\F,L_2)
\sqrt{\frac{J(z,\F,L_2)}{\d_n}\vee 1},
\end{align*}
for $1/\d_n=\sqrt n\|F\|_{P,2}$.
Here we have used the subadditivity of the map $\d\mapsto J(\d,\F,L_2)$, 
and the inequality
$J(C\d,\F,l_2)\le CJ(\d,\F,L_2)$ for $C\ge 1$ in the last step. 
We can bound
the sum of the three terms on the right side by 
a multiple of the maximum of these terms and conclude that
the left side is smaller than at least one of the three terms. Solving
next yields that
$$J(z,\F,L_2)
\lesssim J(\d,\F,L_2)
\vee  \frac{J^2(\d_n,\F,L_2)}{\d_n}
\vee J(\d_n,\F,L_2).$$
Because $J(\d_n,\F,L_2)\ge \d_n$ for 
every $\d_n>0$, by the definition of the entropy integral,
the third term on the right
is bounded by the second term.
We substitute the bound in (\ref{EqHulp}) to finish the proof.
\end{proof}

\section{Unbounded Classes}
In this section we investigate
relaxations of the assumption that
the class $\F$ of functions is uniformly bounded, made
in Theorem~\ref{TheoremGineKoltchinskiiAbstract}. We start
with a moment bound on the envelope.

\begin{theorem}
\label{TheoremGineKoltchinskiiLp}
Let $\F$ be a $P$-measurable class of measurable functions 
with envelope function $F$ such that $PF^{(4p-2)/(p-1)}<\infty$ for some $p>1$
and such that $\F^2$ and $\F^4$ are $P$-measurable.
If $Pf^2< \d^2PF^2$ for every $f$ and some $\d\in(0,1)$, then
$$\E_P^*\|\ep_n\|_\F
\lesssim J\bigl(\d,\F,L_2\bigr)
\Bigl(1+\frac{J(\d^{1/p},\F,L_2)}{\d^2\sqrt n}
\frac{\|F\|_{P,(4p-2)/(p-1)}^{2-1/p}}{\|F\|_{P,2}^{2-1/p}}\Bigr)^{p/(2p-1)}
\|F\|_{P,2}.$$
\end{theorem}

\begin{proof}
Application of (\ref{EqRef2.14.1}) to the functions $f^2$, forming the
class $\F^2$ with envelope function $F^2$, yields 
\begin{equation}
\label{EqBoundF2}
\E^*_P\|\ep_n\|_{\F^2}\lesssim \E_P^*J\Bigl(\frac{\s_{n,4}^2}
{(\PP_nF^4)^{1/2}},\F^2\given F^2,L_2\Bigr)\,(\PP_nF^4)^{1/2},
\end{equation}
for $\s_{n,r}$ the diameter of $\F$ in $L_r(\PP_n)$, i.e.\
\begin{equation}
\label{EqDiameter}
\s_{n,r}^r=\sup_f\PP_n|f|^r.
\end{equation}
Preservation properties of uniform entropy (see \cite{Pollardbook}, or
\cite{vdVW}, 2.10.20, where the supremum over $Q$ can also be
moved outside the integral to match our current definition of
entropy integral, applied to $\phi(f)=f^2$ with $L=2F$) show that
$J(\d,\F^2\given F^2,L_2)\lesssim J(\d,\F\given F,L_2)$, for every $\d>0$.
Because $\PP_nf^2= Pf^2+n^{-1/2}\ep_nf^2$ and $Pf^2\le \d^2PF^2$ by assumption,
we find that
\begin{align}
\E_P^*\s_{n,2}^2
&\lesssim\d^2PF^2+\frac1{\sqrt n}
\E_P^*J\Bigl(\frac{\s_{n,4}^2}
{(\PP_nF^4)^{1/2}},\F,L_2\Bigr)\,(\PP_nF^4)^{1/2}.
\label{Eq***}
\end{align}
The next step is to bound $\s_{n,4}$ in terms of $\s_{n,2}$.

By H\"older's inequality, for any conjugate pair $(p,q)$ and any
$0<s<4$,
$$\PP_nf^4\le\PP_n |f|^{4-s}F^s\le 
\bigl(\PP_n|f|^{(4-s)p}\bigr)^{1/p}\bigl(\PP_nF^{sq}\bigr)^{1/q}.$$
Choosing $s$ such that $(4-s)p=2$ (and hence
$sq=(4p-2)/(p-1)$), we find that 
$$\s_{n,4}^{4}\le \s_{n,2}^{2/p}\bigl(\PP_nF^{sq}\bigr)^{1/q}.$$
We insert this bound in (\ref{Eq***}).
The function $(x,y)\mapsto x^{1/p}y^{1/q}$ is concave,
and hence the function $(x,y,z)\mapsto J(\sqrt{x^{1/p}y^{1/q}/z},\F,L_2)\sqrt z$
can be seen to be concave by the same arguments as in the
proof of Theorem~\ref{TheoremGineKoltchinskiiAbstract}.
Therefore, we can apply Jensen's inequality to see that 
$$\E_P^*\s_{n,2}^2\lesssim\d^2PF^2+\frac1{\sqrt n}
J\Bigl(\frac{(\E^*_P\s_{n,2}^2)^{1/(2p)}\bigl(PF^{sq}\bigr)^{1/(2q)}}
{(PF^4)^{1/2}},\F,L_2\Bigr)\,(PF^4)^{1/2}.$$
We conclude that $z:=(\E_P^*\s_{n,2}^2)^{1/2}/\|F\|_{P,2}$ satisfies
\begin{align*}
z^2
&\lesssim \d^2 
+\frac1{\sqrt n} J\Bigl(z^{1/p}\frac{(PF^2)^{1/(2p)}(PF^{sq})^{1/(2q)}}
{(PF^4)^{1/2}},\F,L_2\Bigr)\frac{(PF^4)^{1/2}}{PF^2}\\
&\lesssim \d^2 
+J(z^{1/p},\F,L_2)\frac{(PF^{sq})^{1/(2q)}}{\sqrt n(PF^2)^{1-1/(2p)}}.
\end{align*}
In the last step we use that $J(C\d,\F,L_2)\le CJ(\d,\F,L_2)$ for $C\ge1$,
and H\"older's inequality as previously to see that the present $C$ 
satisfies this condition.
We next apply Lemma~\ref{LemmaRecursion} (with $r=1/p$) to obtain a bound
on $J(z,\F,L_2)$, and conclude the proof by substituting this
bound in (\ref{EqHulp}).
\end{proof}

The preceding theorem assumes only a finite moment of
the envelope function, but in comparison to
Theorem~\ref{TheoremGineKoltchinskiiAbstract} 
substitutes $J(\d^{1/p},\F,L_2)$ in the correction term of
the upper bound, where $p>1$ and hence $\d^{1/p}\gg \d$ for small
$\d$. In applications
to moduli of continuity of minimum contrast criteria this
is sufficient to obtain consistency with a rate, but
typically the rate will be suboptimal. The rate improves
as $p\da1$, which requires finite moments of the envelope
function of order increasing to infinity, the limiting
case $p=1$ corresponding to a bounded envelope, as in 
Theorem~\ref{TheoremGineKoltchinskiiAbstract}. The following theorem
interpolates between finite moments of any order and
a bounded envelope function. If applied to obtaining
rates of convergence it gives rates that are optimal up to
a logarithmic factor.

\begin{theorem}
\label{TheoremGineKoltchinskiiExp}
Let $\F$ be a $P$-measurable class of measurable functions 
with envelope function $F$ such that $P\exp(F^{p+\r})<\infty$ for some $p,\r>0$
and such that $\F^2$ and $\F^4$ are $P$-measurable.
If $Pf^2< \d^2PF^2$ for every $f$ and some $\d\in(0,1/2)$, then
for a constant $c$ depending on $p$, $PF^2$, $PF^4$ and 
$P\exp(F^{p+\r})$,
$$\E_P^*\|\ep_n\|_\F
\le c J\bigl(\d,\F,L_2\bigr)
\Bigl(1+\frac{J\bigl(\d(\log(1/\d))^{1/p},\F,L_2\bigr)}{\d^2\sqrt n}\Bigr).$$
\end{theorem}

\begin{proof}
Fix $r=2/p$.
The functions $\psi,\overline\psi:[0,\infty)\to[0,\infty)$ defined by
$$\psi(f)=\log^r(1+f),\qquad
\overline\psi(f)=e^{f^{1/r}}-1,$$
are each other's inverses, and are increasing from 
$\psi(0)=\overline\psi(0)=0$ to infinity.
Thus their primitive functions $\Psi(f)=\int_0^f\psi(s)\,ds$
and $\overline\Psi(f)=\int_0^f\overline\psi(s)\,ds$ satisfy
\emph{Young's inequality} $fg\le \Psi(f)+\overline\Psi(g)$, for every $f,g\ge0$
(e.g.\ \cite{Boyd}, page 120, 3.38).

The function $t\mapsto t\log^r (1/t)$ is concave in a neighbourhood
of 0 (specifically: on the interval $(0,e^{1-r}\wedge 1)$), with limit
from the right equal to 0 at 0, and derivative tending to infinity 
at this point. Therefore, there exists a \emph{concave, increasing} function 
$k:(0,\infty)\to (0,\infty)$ that is identical to 
$t\mapsto t\log^r (1/t)$ near 0 and bounded below and above by a
positive constant times the identity throughout its domain. 
(E.g.\ extend $t\mapsto t\log^r (1/t)$ linearly with slope 1
from the point where the derivative of the latter function has decreased
to 1.) Write $k(t)=t\ell^r(t)$, so that $\ell^r$ is bounded
below by a constant and $\ell(t)=\log (1/t)$ near 0.
Then,  for every $t>0$,
\begin{equation}
\label{EqBoundLog}
\frac{\log(2+t/C)}{\ell(C)}\lesssim \log(2+t).
\end{equation}
(The constant in $\lesssim$ may depend on $r$.)
To see this, note that for $C>c$
the left side is bounded by a multiple of $\log(2+t/c)$, whereas
for small $C$ the left side is bounded by a multiple of 
$\bigl[\log(2+t)+\log(1+1/C)\bigr]/\ell(C)\lesssim \log(2+t)+1$.

From the inequality $\Psi(f)\le f\psi(f)$, we obtain that, for $f>0$,
$$\Psi\Bigl(\frac f{\log^r(2+f)}\Bigr)\lesssim f.$$
Therefore, by (\ref{EqBoundLog}) followed by Young's inequality, 
\begin{align*}
\frac{f^4}{k(C^2)}
&=\frac{f^2/C^2}{\log^r(2+f^2/C^2)}
\frac{f^2\log^r(2+f^2/C^2)}{\ell^r(C^2)}\\
&\lesssim\frac{f^2}{C^2}+\overline\Psi\bigl(F^2\log^r(2+F^2)\bigr).
\end{align*}
On integrating this with respect to the empirical measure,
with $C^2=\PP_nf^2$, we see that, with 
$G=\overline\Psi\bigl(F^2\log^r(2+F^2)\bigr)$,
$$\PP_nf^4\lesssim k(\PP_nf^2)\,\bigl(1+\PP_nG\bigr).$$
We take the supremum over $f$ to bound $\s_{n,4}^4$ 
as in (\ref{EqDiameter}) in terms
of $k(\s_{n,2}^2)$, and next substitute this bound
in (\ref{Eq***}) to find that
\begin{align*}
\E_P^*\s_{n,2}^2
&\le\d^2PF^2+\frac1{\sqrt n}
\E_P^*J\Bigl(\frac{\sqrt {k(\s_{n,2}^2)}\sqrt{1+\PP_nG}}
{(\PP_nF^4)^{1/2}},\F,L_2\Bigr)(\PP_nF^4)^{1/2}\\
&\le\d^2PF^2+\frac1{\sqrt n}
J\Bigl(\frac{\sqrt {k(\E_P^*\s_{n,2}^2)}\sqrt{1+PG}}
{(PF^4)^{1/2}},\F,L_2\Bigr)(PF^4)^{1/2},
\end{align*}
where we have used the concavity of $k$, and
the concavity of the other maps, as previously. 
By assumption the expected value $PG$ is finite for $r=2/p$.
It follows that
$z^2=\E_P^*\s_{n,2}^2/PF^2$ satisfies, for suitable 
constants $a,b,c$ depending on $r$, $PF^2$, $PF^4$ and $PG$,
$$z^2\lesssim \d^2+\frac{a}{\sqrt n}J(\sqrt{k(z^2b)}c,\F,L_2).$$
By concavity and the fact that $k(0)=0$, we have
$k(Cz)\le Ck(z)$, for $C\ge 1$ and $z>0$. The function
$z\mapsto \sqrt{k(z^2b)}c$ inherits this property.
Therefore we can apply Lemma~\ref{LemmaRecursionLog}, 
with $k$ of the lemma equal to the present function 
$z\mapsto \sqrt{k(z^2b)}c$, to obtain a bound on $J(z,\F,L_2)$
in terms of $J(\d,\F,L_2)$ and $J\bigl(\sqrt{k(\d^2b)c},\F,L_2)$,
which we substitute in (\ref{EqHulp}).
Here $k(\d^2)=\d^2\log^r(1/\d)$ for sufficiently small $\d>0$ and 
$\k(\d^2)\lesssim \d^2\lesssim\d^2\log^r(1/\d)$ for $\d<1/2$ 
and bounded away from 0.
Thus we can simplify the bound to the one in the statement
of the theorem, possibly after increasing the constants
$a,b,c$ to be at least 1, to complete the proof.
\end{proof}

\begin{lemma}
\label{LemmaRecursionLog}
Let $J:(0,\infty)\to\RR$ be a concave, nondecreasing function with $J(0)=0$,
and let $k:(0,\infty)\to(0,\infty)$ be nondecreasing and 
satisfy $k(Cz)\le C k(z)$ for $C\ge 1$ and $z>0$.
If $z^2\le A^2+B^2 J\bigl(k(z)\bigr)$ for some $A,B>0$, then
$$J(z)\lesssim J(A)\Bigl[1+J\bigl(k(A)\bigr)\Bigl(\frac BA\Bigr)^2\Bigr].$$
\end{lemma}

\begin{proof}
As noted in the proof of Lemma~\ref{LemmaRecursion}
the properties of $J$ imply that
$J(Cz)\le C J(z)$ for $C\ge1$ and any $z>0$. In view
of the assumed property of $k$ and the monotonicity of $J$ it follows that
$J\circ k(C z)\le CJ\circ k(z)$ for every $C\ge 1$ and $z>0$.
Therefore, by the monotonicity of $J$ and $k$, and the assumption on $z$,
$$J\circ k(z)
\le J\circ k\Bigl(\sqrt{A^2+B^2J\circ k(z)}\Bigr)
\le J\circ k(A)\sqrt{1+(B/A)^2J\circ k(z)}.$$
As in the proof of Lemma~\ref{LemmaRecursion}
we can solve this for $J\circ k(z)$ to find that
$$J\circ k(z)\lesssim J\circ k(A)+J\circ k(A)^2\Bigl(\frac BA\Bigr)^2.$$
Next again by the monotonicity of $J$,
\begin{align*}
J(z)&\le J\Bigl(\sqrt{A^2+B^2J\circ k(z)}\Bigr)
\le J(A)\sqrt{1+(B/A)^2J\circ k(z)}\\
&\lesssim J(A)\Bigl[1+\Bigl(\frac BA\Bigr)\sqrt{J\circ k(A)}
+\Bigl(\frac BA\Bigr)^2J\circ k(A)\Bigr].
\end{align*}
The middle term on the right side is bounded by 
the sum of the first and third terms.
\end{proof}

\bibliographystyle{acm}
\bibliography{maxUE}

\end{document}